\documentclass{amsart}

\usepackage{amsmath,amssymb,amsthm}
\usepackage{graphicx,tikz}
\newtheorem{theorem}{Theorem}

\newtheorem{lemma}{Lemma}

\theoremstyle{definition}

\theoremstyle{remark}

\begin{document}

\title[]{On the Wasserstein Distance between Classical Sequences and the Lebesgue Measure}
\keywords{Low discrepancy sequence, Wasserstein distance, Integration Error, Exponential Sum, Erd\H{o}s-Tur\'an inequality, Irregularities of Distribution, Approximation Theory.}
\subjclass[2010]{11L07, 41A25, 42B05, 65D30.}

\author[]{Louis Brown}
\address{Department of Mathematics, Yale University, New Haven, CT 06511, USA}
\email{louis.brown@yale.edu}

\author[]{Stefan Steinerberger}
\address{Department of Mathematics, Yale University, New Haven, CT 06511, USA}
\email{stefan.steinerberger@yale.edu}

\thanks{This paper is part of L.B.'s PhD thesis, he gratefully acknowledges support from Yale Graduate School. S.S. is partially supported by the NSF (DMS-1763179) and the Alfred P. Sloan Foundation.}

\begin{abstract} We discuss the classical problem of measuring the regularity of distribution of sets of $N$ points in $\mathbb{T}^d$. A recent line of investigation is to study the cost ($=$ mass $\times$ distance) necessary to move Dirac measures placed on these points to the uniform distribution. We show that Kronecker sequences satisfy optimal transport distance in $d \geq 2$ dimensions. This shows that for differentiable $f: \mathbb{T}^d \rightarrow \mathbb{R}$ and badly approximable vectors $\alpha \in \mathbb{R}^d$, we have
$$ \left|  \int_{\mathbb{T}^d} f(x) dx - \frac{1}{N} \sum_{k=1}^{N} f(k \alpha) \right| \leq c_{\alpha}  \frac{ \| \nabla f\|^{(d-1)/d}_{L^{\infty}}\| \nabla f\|^{1/d}_{L^{2}}  }{N^{1/d}}.$$
We note that the result is uniform in $N$ (it holds for a sequence instead of a set). Simultaneously, it refines the classical integration error for Lipschitz functions, $\| \nabla f\|_{L^{\infty}} N^{-1/d}$. We obtain a similar improvement for numerical integration with respect to the regular grid.
The main ingredient is an estimate involving Fourier coefficients of a measure; this allows for existing estimates to be conveniently `recycled'. We present several open problems.
\end{abstract}

\maketitle

\vspace{-20pt}

\section{Introduction}
\subsection{Introduction}
We study the problem of measuring the regularity of point sets $\left\{x_1, \dots, x_N\right\} \subset \mathbb{T}^d$ as well as infinite sequences. There are many classical notions of regularity as well as
good constructions of sets minimizing these notions that have
been proposed; we refer to the classical textbooks \cite{chazelle, dick, drmota, kuipers}. 

\begin{center}
\begin{figure}[h!]
\begin{tikzpicture}[scale=1.2]
\draw [thick] (0,0) -- (10,0);
\filldraw (0,0) circle (0.04cm);
\draw [very thick] (0,0) -- (0,0.2);
\filldraw (10/29,0) circle (0.05cm);
\draw [very thick] (10/29,0) -- (10/29,0.4);
\filldraw (40/29,0) circle (0.05cm);
\draw [very thick] (40/29,0) -- (40/29,0.4);
\filldraw (50/29,0) circle (0.05cm);
\draw [very thick] (50/29,0) -- (50/29,0.4);
\filldraw (60/29,0) circle (0.05cm);
\draw [very thick] (60/29,0) -- (60/29,0.4);
\filldraw (70/29,0) circle (0.05cm);
\draw [very thick] (70/29,0) -- (70/29,0.4);
\filldraw (90/29,0) circle (0.05cm);
\draw [very thick] (90/29,0) -- (90/29,0.4);
\filldraw (130/29,0) circle (0.05cm);
\draw [very thick] (130/29,0) -- (130/29,0.4);
\filldraw (160/29,0) circle (0.05cm);
\draw [very thick] (160/29,0) -- (160/29,0.4);
\filldraw (200/29,0) circle (0.05cm);
\draw [very thick] (200/29,0) -- (200/29,0.4);
\filldraw (220/29,0) circle (0.05cm);
\draw [very thick] (220/29,0) -- (220/29,0.4);
\filldraw (230/29,0) circle (0.05cm);
\draw [very thick] (230/29,0) -- (230/29,0.4);
\filldraw (240/29,0) circle (0.05cm);
\draw [very thick] (240/29,0) -- (240/29,0.4);
\filldraw (250/29,0) circle (0.05cm);
\draw [very thick] (250/29,0) -- (250/29,0.4);
\filldraw (280/29,0) circle (0.05cm);
\draw [very thick] (280/29,0) -- (280/29,0.4);
\end{tikzpicture}
\caption{The renormalized quadratic residues in $\mathbb{F}_{29}$ rescaled to $[0,1]$. Every dot except the one at zero represents two quadratic residues corresponding to two Dirac delta measures. How costly is it to move this point measure to the uniform distribution on $[0,1]$?}
\end{figure}
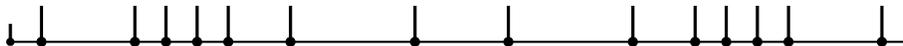
\end{center}

\vspace{-10pt}

The classical theory has developed a useful machinery in terms of exponential sums that exploits regularities of number-theoretic constructions. We will not, initially, pursue this
path and instead ask a different question: consider the measure
$$ \mu = \frac{1}{N} \sum_{k=1}^{N} \delta_{x_k}.$$
How would we go about distributing this measure in such a way that the end result is the Lebesgue measure on $\mathbb{T}^d$?  Here, the `cost' of transporting $\delta$ units of measure across
distance $d$ is understood to be $\delta \cdot d$. An even more practical example is the following: suppose we have people evenly distributed over $\mathbb{T}^d$ and $N$ supermarkets placed in
$\left\{x_1, \dots, x_N\right\} \subset \mathbb{T}^d$. Demand and supply are exactly matched: how far would the trucks have to drive to distribute the goods from the supermarkets evenly? This is Monge's transportation
problem from 1781. It is easy to see that
$$ \mbox{transportation cost} \geq c_d N^{-1/d},$$
where $c_d$ is a universal constant depending only on the dimension (see \S 1.2 for a sketch of the argument). This scaling is, for example, assumed by a rescaling of $\mathbb{Z}^d$ intersected with $\mathbb{T}^d \cong [0,1]^d$. Our paper was motivated by the following questions
\begin{enumerate}
\item Do the classical constructions of regular sequences in $\mathbb{T}^d$ from  \cite{chazelle, dick, drmota, kuipers} have an optimal transportation cost?  Do they have it uniformly in $N$?
\item How does one go about proving such results?
\item Does this perspective lead to new results?
\end{enumerate}
We emphasize that these types of problems, estimating transport cost from one measure to another, have been actively investigated in the field of Optimal Transport \cite{santa, villani}. Here, the emphasis is usually on existence and uniqueness of optimal transport maps as well as fine qualitative and quantitative properties. Many special cases have been actively investigated in probability theory, we emphasize 
the problems of estimating the transport of random points to the Lebesgue measure, more generally, random points drawn from a measure $\mu$ to $\mu$ or random points to random points
\cite{ot1, am, ot6, ot5, ot3, ot2, ot7, ot4}. As far as we know, special structures arising from Number Theory or Combinatorics have not been considered before (however, there are some interesting precursors in \cite{blum, louis, hensley, stein0, stein3, su, su1, su2}).

\subsection{Setup.}

We introduce the $p-$Wasserstein distance (the example above had $p=1$ and is also known as the `Earth Mover distance') between two measures
$\mu$ and $\nu$ as
$$ W_p(\mu, \nu) = \left( \inf_{\gamma \in \Gamma(\mu, \nu)} \int_{M \times M}{ |x-y|^p d \gamma(x,y)}\right)^{1/p},$$
where $| \cdot |$ is the usual distance on the torus and $\Gamma(\mu, \nu)$ denotes the collection of all measures on $M \times M$
with marginals $\mu$ and $\nu$, respectively (also called the set of all couplings of $\mu$ and $\nu$). 
Our two measures under consideration are 
$$ \mu = \frac{1}{N} \sum_{k=1}^{N}{\delta_{x_k}} \qquad \mbox{and} \qquad \nu = dx,$$
where $dx$ refers to the normalized volume measure. It is relatively easy to see that, we have an
(optimal) lower bound that is independent of the set $\left\{x_1, \dots, x_N\right\} \subset \mathbb{T}^d$
$$W_1\left( \frac{1}{N} \sum_{k=1}^{N}{\delta_{x_k}}, dx\right) \geq \frac{c_d}{N^{1/d}}.$$
The short argument is as follows: let $B_{\varepsilon N^{-1/d}}(x)$ be a ball of radius $\varepsilon N^{-1/d}$ centered
around $x$. We note that the total volume of $N$ such balls satisfies
$$ \left| \bigcup_{i=1}^{N}{ B_{\varepsilon N^{-1/d}}(x)} \right| \leq N \left| B_{\varepsilon N^{-1/d}}(x) \right| \leq \omega_d \varepsilon^d$$
for some universal constant $\omega_d$ depending only on the dimension. For $\varepsilon$ sufficiently small (depending only on $\omega_d$), this is much less than the
volume of $\mathbb{T}^d$ and therefore most of the Lebesgue measure on $\mathbb{T}^d$ is at distance $\gtrsim \varepsilon N^{-1/d}$ from
the set $\left\{x_1, \dots, x_N\right\}$. Here and in what follows, we use $A \lesssim B$ to denote the existence of a universal constant $c>0$ such that $A \leq cB$. 
We refer to Santamborigo \cite{santa} or Villani \cite{villani} for introductions to Optimal Transport and the Wasserstein distance.

\subsection{Existing Results in One Dimension.} There are several recent results in the one-dimensional setting. 
Given a finite set on the one-dimensional torus $\left\{x_1, \dots, x_N\right\} \subset \mathbb{T}$, we associate to it the measure 
$$ \mu = \frac{1}{N} \sum_{k=1}^{N}{\delta_{x_k}}.$$
A natural quantity that is frequently studied (see e.g. \cite{dick, drmota, kuipers}) is the discrepancy
$$ D_N(\mu) = \sup_{J \subset \mathbb{T} \atop J~\mbox{\tiny interval}}{ | \mu(J) - |J| |}.$$
It is easy to see that $N^{-1} \leq D_N \leq 1$. The inequality
$$ W_1(\mu, dx) \lesssim D_N(\mu)$$
follows from Monge-Kantorovich duality (this is carried out in greater detail in \cite{blum} or \cite{stein}). Another notion of regularity is Zinterhof's diaphony \cite{drmota, zinterhof}, which can be defined as
\begin{align*}
 F_N(\mu) = \left( \sum_{k \in \mathbb{Z} \atop k \neq 0}{ \frac{| \widehat{\mu}(k)|^2}{k^2}}\right)^{1/2}.
 \end{align*}
 One of the key points of our paper is that we are able to generalize Zinterhof's diaphony to higher dimensions.
A recent inequality of Peyre \cite{peyre} can be reinterpreted as saying (see \cite{stein0, stein3}) that
$$ W_2(\mu, dx) \lesssim F_N(\mu).$$
Summarizing, we have two inequalities and Holder's inequality
\begin{align*}
 W_1(\mu, dx) &\lesssim D_N(\mu)\\
    W_1(\mu, dx)  \leq W_2(\mu, dx) &\lesssim F_N(\mu) 
   \end{align*}
For classical one-dimensional constructions in Number Theory, the notions $D_N$ and $F_N$ have been studied intensively.
This connection immediately implies a series of results for the Wasserstein distance: the upper bounds that we obtain for
the $W_2$ distance are better, by a factor of $(\log{N})^{1/2}$, than the estimate on $D_N$. A simple example is given by
the van der Corput sequence in base $r \in \mathbb{N}$ (see e.g. \cite{dick}). The element $x_n$ is given by writing $n$ in base $r$, inverting
the digits at the comma and then reinterpreting this as a real number; the van der Corput sequence in base 2 starts with
$0.5, 0.25, 0.75, 0.125, 0.625$ and so on. It is known to satisfy $D_N \lesssim_r N^{-1} \log{N}$. Using an existing result
of Proinov \& Grozdanov \cite{proinov}, we can obtain the following improved estimate on the transport distance.

\begin{theorem}[Proinov \& Grozdanov \cite{proinov}] Let $(x_n)_{n=1}^{\infty}$ denote the van der Corput sequence in base $r$. Then, uniformly in $N$,
$$ W_2\left( \frac{1}{N} \sum_{k=1}^{N}{\delta_{x_k}}, dx \right) \lesssim_r \frac{(\log{N})^{1/2}}{N}.$$
\end{theorem}
A recent result of Graham \cite{graham} shows that this is the optimal rate: Graham showed that for any sequence
$(x_n)_{n=1}^{\infty}$ in $[0,1]$, we have, for some universal $c>0$,
$$ W_2\left( \frac{1}{N} \sum_{k=1}^{N}{\delta_{x_k}}, dx \right) \geq c \frac{(\log{N})^{1/2}}{N} \qquad \mbox{for infinitely many values of}~N.$$
The inequality $F_N \gtrsim W_2(\mu, dx)$, a result of Peyre \cite{peyre}, implies the same result for the Zinterhof diaphony which recovers a result of Proinov
 \cite{proinov0}. A natural question is whether this rate of growth is attained by other sequences as well.
The second author recently remarked \cite{stein0} that the $(n\alpha)-$sequence satisfies a similar growth. Moreover, quadratic residues of a finite field, suitably rescaled,
behave better than one would obtain using the Polya-Vinogradov estimate (see for example \cite{burgess}).
\begin{theorem}[$n\alpha$ Sequence and Quadratic Residues, \cite{stein0}] Let $\alpha$ be badly approximable and $x_n = \left\{n \alpha \right\}$, then
$$ W_2\left( \frac{1}{N} \sum_{k=1}^{N}{\delta_{x_k}}, dx \right) \lesssim_{\alpha} \frac{(\log{N})^{1/2}}{N}.$$
Moreover, let $p$ be a prime and let $x_k = \left\{ k^2/p \right\}$ for $ 1 \leq k \leq p$. Then
$$ W_2\left( \frac{1}{p} \sum_{k=1}^{p}{\delta_{x_k}}, dx \right) \lesssim \frac{1}{\sqrt{p}}.$$
\end{theorem}

This connection between the Wasserstein distance, Diaphony, the Sobolev space $\dot H^{-1}$ and the corresponding exponential sum estimate does not seem
to have been noticed before the paper \cite{stein0}. For that reason, we believe that there are many interesting results in $d=1$ that are within reach.

\begin{figure}[h!]
\begin{center}
\begin{tikzpicture}[scale=3]
\draw [very thick] (0,0) -- (1,0) -- (1,1) -- (0,1) -- (0,0);
\filldraw (1/6, 1/6) circle (0.02cm);
\filldraw (3/6, 1/6) circle (0.02cm);
\filldraw (5/6, 1/6) circle (0.02cm);
\filldraw (1/6, 3/6) circle (0.02cm);
\filldraw (3/6, 3/6) circle (0.02cm);
\filldraw (5/6, 3/6) circle (0.02cm);
\filldraw (1/6, 5/6) circle (0.02cm);
\filldraw (3/6, 5/6) circle (0.02cm);
\filldraw (5/6, 5/6) circle (0.02cm);
\draw [thick] (1/3,0) -- (1/3,1);
\draw [thick] (2/3,0) -- (2/3,1);
\draw [thick] (0,1/3) -- (1,1/3);
\draw [thick] (0,2/3) -- (1,2/3);
\end{tikzpicture}
\end{center}
\caption{A distribution with small Wasserstein transportation cost -- however, these constructions are not uniform in $N$.}
\end{figure}
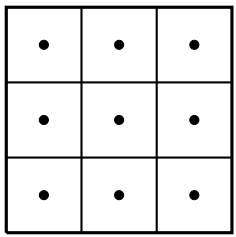

\subsection{Existing Results in Higher Dimensions}
It is easy to see that for any fixed set of points $\left\{x_1, \dots, x_N\right\} \subset \mathbb{T}^d$, the lattice construction (see Fig. 2) is
optimal up to constants. However, if one were to construct an infinite sequence $(x_n)_{n=1}^{\infty}$ with estimates that are uniformly
good, a lattice construction does not seem to be particularly useful; see Fig 2: where would one put the next point and the point after that?
A general result has recently been obtained by the authors  \cite{louis} on general compact manifolds. If $(M,g)$ is a compact manifold
without boundary and $G(\cdot, \cdot)$ denotes the Green's function of the Laplacian $-\Delta_g$, then the greedy construction
$$ x_n = \arg\min_{x \in M} \sum_{k=1}^{n-1}{G(x,x_k)}.$$
has good distribution properties.
\begin{theorem}[B. \& S. \cite{louis}] Let $x_n$ be a sequence obtained in this way on a compact $d-$dimensional manifold. Then 
$$ W_2\left( \frac{1}{N} \sum_{k=1}^{N} \delta_{x_k} , dx \right) \lesssim_M \begin{cases}N^{-1/2} (\log{N})^{1/2} \qquad &\mbox{if}~d=2 \\
N^{-1/d} &\mbox{if}~d \geq 3. \end{cases}$$
\end{theorem}
It is not clear whether the $(\log{N})^{1/2}$ factor is necessary. The crucial ingredient to these types of results is a smoothing procedure
introduced in \cite{stein0} coupled with Peyre's estimate \cite{peyre}. This allows us to obtain a general bound which is particularly simple
on $\mathbb{T}^d$ and reads as follows:
$$\boxed{ W_2\left( \frac{1}{N} \sum_{k=1}^{N} \delta_{x_k} , dx \right)^2 \lesssim_d \inf_{t > 0} \left[ t + \sum_{k \in \mathbb{Z}^d \atop k \neq 0} \frac{e^{-\|k\|^2 t}}{\|k\|^2} \left| \frac{1}{N} \sum_{n=1}^{N} e^{2\pi i \left\langle k, x_n \right\rangle} \right|^2 \right] }$$
This inequality has a series of remarkable features:
\begin{enumerate}
\item It is phrased exclusively in terms of exponential sums that have been well studied for a variety of sequences; in particular, information about the size of these exponential sums is available for many sequences.
\item The quantity on the right-hand side reduces to the notion of diaphony $F_N$ in the one-dimensional case $d=1$ and $t=0$.
\item However, in contrast to classical diaphony, the quantity is \textit{finite} for \textit{any} set of points and \textit{any} dimension $d \in \mathbb{N}$ for all $t > 0$. It can thus be regarded as 
a useful generalization of Zinterhof's diaphony.
\end{enumerate}
We note that diaphony has been studied in a variety of settings \cite{dia1, dia2, dia3, dia4, zinterhof, zinterhof2}. We interpret it, in one dimension, as the quantity
$$ F_N = \left\|  \frac{1}{N} \sum_{k=1}^{N} \delta_{x_k}  \right\|_{\dot H^{-1}},$$
where $\dot H^{-1}$ is the Sobolev space. We observe that this quantity becomes meaningless in dimensions $d \geq 2$ because Dirac deltas are no longer contained in the Sobolev
space $\dot H^{-1}$ (or, put differently, the infinite sums do not converge). This has been a persistent issue in trying to define notions of discrepancy in higher dimensions on other geometries (see e.g. 
Freeden \cite{freeden} or Grabner, Klinger \& Tichy \cite{grabner}). In contrast, we can rewrite our inequality (even on general manifolds) as 
$$W_2\left( \frac{1}{N} \sum_{k=1}^{N} \delta_{x_k} , dx \right) \lesssim_M \inf_{t > 0} \left[ \sqrt{t} +  \left\| e^{t\Delta}  \frac{1}{N} \sum_{k=1}^{N} \delta_{x_k}  \right\|_{\dot H^{-1}}\right],$$
where $e^{t\Delta}$ is the heat propagator, i.e. the forward evolution of the heat equation. This quantity is always finite for any $t > 0$. We believe this to be an insight that might be useful
in discrepancy theory as a suitable generalization of diaphony to higher dimensions. We also note that this notion is intimately tied to the integration error for Lipschitz functions, see \S 2.3. below.

\section{Main Results}
\subsection{A Random Walk.}
We have already mentioned a series of results for $d=1$. We add another one to the list: here,  we do not consider a sequence of points but a sequence of probability measures. Let $\mu_k$ be the measure that arises from an unbiased random walk on $\mathbb{T} \cong [0,1]$ where each step is $\pm \alpha$ (independently and with likelihood $1/2$ each) and $\alpha$ is a quadratic irrational. This model was studied by Su \cite{su} (see also Hensley \& Su \cite{hensley} and Su \cite{su2}). The main result in \cite{su} showed that the measure arising after $k$ random steps satisfies
$$ D_N(\mu_k) \lesssim_{\alpha} k^{-1/2}.$$
We note that this result immediately implies $W_1(\mu_k, dx) \lesssim k^{-1/2}$. Here, we show that for this model we can obtain a (worse) bound for the (larger) $W_2-$distance.
\begin{theorem} We have
$$ W_2(\mu_k, dx) \lesssim_{\alpha} k^{-1/4}.$$
\end{theorem}
We emphasize that the framework discussed in this paper enables us to reduce Theorem 4 to standard estimates. Theorem 4 is presumably not optimal and stronger results should be true.
Hensley \& Hu \cite{hensley} discuss their result and put it in direct relation to the Wasserstein distance. We hope that our approach will be a useful technique for these types of problems.


\subsection{Kronecker sequences.} Kronecker sequences are the natural generalization of $\left\{n \alpha\right\}$ sequences (irrational rotations) on $\mathbb{T}$. We say that
a vector $(\alpha_1, \alpha_2, \dots, \alpha_d) \in \mathbb{R}^d$ is badly approximable if, for all positive integers $q \neq 0$, we have
$$ \max_{1 \leq j \leq d}{ \| \alpha_j q\|} \geq \frac{c_{\alpha}}{q^{1/d}},$$
where $\| \cdot \|$ is the distance to the nearest integer.
 By Dirichlet's approximation theorem, this is the optimal scaling. The existence of such a vector follows from continued fraction expansion
when $d=1$. The first examples in higher dimensions are due to Perron \cite{perron}, Davenport \cite{davenport} showed that there are uncountably many such vectors for $d=2$ and Schmidt \cite{sch}
extended this result to $d \geq 3$. The Kronecker sequence is then defined via 
$$ x_n = (n \alpha_1, n \alpha_2, \dots, n \alpha_d) \quad \mbox{mod}~1,$$
where $\mbox{mod}~1$ is to be interpreted component-wise. We now establish that these sequences have uniformly good transport properties to the uniform measure.
\begin{theorem} Let $d \geq 2$ and let $\alpha \in \mathbb{R}^d$ be badly approximable. Then the Kronecker sequence satisfies
$$ W_2\left(\frac{1}{N} \sum_{k=1}^{N}{\delta_{x_k}}, dx\right) \lesssim_{c_{\alpha}, d} N^{-1/d}
$$
\end{theorem}

We emphasize that this result is best possible (up to constants) as well as \textit{uniform} in $N$. It is not at all clear to us whether the conditon of $\alpha$ being badly approximable is necessary; however, in
light of results in $d=1$, this is quite conceivable. 

\subsection{Integration.}
Let us consider the problem of numerically integrating a function $f:\mathbb{T}^d \rightarrow \mathbb{R}$ which we assume to be Lipschitz. It is a classic 1959 result of Bakhvalov \cite{bak} (see also Novak \cite{novak}) that
there are points $(x_k)_{k=1}^{N}$ such that for all differentiable functions $f:\mathbb{T}^d \rightarrow \mathbb{R}$ with Lipschitz constant $\|\nabla f\|_{L^{\infty}}$
 $$ \left|  \int_{\mathbb{T}^d} f(x) dx - \frac{1}{N} \sum_{k=1}^{N} f(x_k) \right| \leq c_d \| \nabla f\|_{L^{\infty}} N^{-1/d}$$
and that this result is optimal in the power of $N$ and its dependance on the Lipschitz constant $\| \nabla f\|_{L^{\infty}}$: there
are functions $f$ for which the error is at that scale (up to constants).  Recently, Hinrichs, Novak, Ullrich and Wozniakowski \cite{hinrichs} (see also
\cite{hinrichs0, hinrichs00}) established rather precise estimates on the constant $c_d$ and showed that product rules (regular grid structures) are a good choice whenever the number of points
$N$ is of the form $N = m^d$. We improve this result in the following manner.

\begin{theorem}  Let $d \geq 2$ and let $\alpha \in \mathbb{R}^d$ be a badly approximable vector. Then, for some universal $c_{\alpha} > 0$ and all differentiable $f: \mathbb{T}^d \rightarrow \mathbb{R}$
$$ \left|  \int_{\mathbb{T}^d} f(x) dx - \frac{1}{N} \sum_{k=1}^{N} f(k \alpha) \right| \leq c_{\alpha} \| \nabla f\|^{(d-1)/d}_{L^{\infty}(\mathbb{T}^d)}\| \nabla f\|^{1/d}_{L^{2}(\mathbb{T}^d)}    N^{-1/d}.$$
\end{theorem}

The main novelties are that:
\begin{enumerate}
\item The results holds uniformly in $N$ along a sequence.
\item The error estimate is actually smaller than the classically assumed dependence on the Lipschitz constant. We note that, trivially
$$\| \nabla f\|^{}_{L^{2}} \leq \| \nabla f\|^{}_{L^{\infty}}$$
which recovers the traditional estimate. At first, this seems like a contradiction to the fact that the dependence on the Lipschitz constant is optimal -- however,
it merely implies that extremal functions for the estimate have to have $\| \nabla f\|^{}_{L^{2}} \sim \| \nabla f\|^{}_{L^{\infty}}$ which is perhaps not surprising (one would
expect them to grow at maximal speed away from the points, so $|\nabla f|$ should be fairly constant). 
\item The result is an explicit improvement in the case where the function $f$ has a large derivative in a small region.
\end{enumerate}

We also emphasize that there is nothing particularly special about the Kronecker sequence: \textit{any} sequence for which we can establish optimal Wasserstein bounds
along the lines outlined above, we will also obtain a version of the integration result; the proof is identical (see \S 3.6. below). Indeed, the result is actually true on general $d-$dimensional manifolds, we refer to Theorem 8 below.
 We also note that the result has similar flavor and scaling as a result that was previously obtained by the second author \cite{stein5} in a different context. If $\alpha \in \mathbb{R}^d$
 is badly approximable, then it is possible to obtain directional Poincar\'{e} inequalities without loss on $\mathbb{T}^d$: for all $f \in C^{\infty}(\mathbb{T}^d)$ with mean value 0, we have
 $$ \| \nabla f\|_{L^2}^{(d-1)/d} \| \left\langle \nabla f, \alpha \right\rangle \|_{L^2}^{1/d} \geq c_{\alpha} \|f\|_{L^2}.$$

\subsection{The Case of the Regular Grid} Let us return to the case of the regular grid (refering to $\mathbb{T}^d$ or $[0,1]^d$ and having $N = m^d$ points that are arranged
as a regular grid). Since we have just improved the classic integration error for the Kronecker sequence, we would expect a similar improvement to hold for the regular grid (which is
well understood to be, in a sense, an optimal set for sampling Lipschitz functions). The classic estimate for a regular grid $(x_n)_{n=1}^{N}$ is
$$ \left|  \int_{[0,1]^d} f(x) dx - \frac{1}{N} \sum_{k=1}^{N} f(x_k) \right| \leq c_d \| \nabla f\|_{L^{\infty}}    N^{-1/d}.$$
Sukharev \cite{suk} (see also \cite{novak2}) has determined the sharp constant.
It is known that `the result cannot be significantly improved for uniformly continuous functions' (Dick \& Pillichshammer \cite[\S 1.3]{dick}). Indeed, there is a corresponding
result of Larcher (unpublished, but see \cite[\S 1.3]{dick}) that shows that the estimate is optimal with regards to modulus of continuity. However, there is an explicit improvement
in terms of $L^p-$spaces that seems to be new.

\begin{theorem} We have, for some explicit constant $c_d$ depending only on the dimension, for all differentiable $f:[0,1]^d \rightarrow \mathbb{R}$ sampled on the regular grid $(x_k)_{k=1}^{N}$ 
$$ \left|  \int_{[0,1]^d} f(x) dx - \frac{1}{N} \sum_{k=1}^{N} f(x_k) \right| \leq c_d \| \nabla f\|^{(d-1)/d}_{L^{\infty}(\mathbb{T}^d)}\| \nabla f\|^{1/d}_{L^{1}(\mathbb{T}^d)}    N^{-1/d}.$$
\end{theorem}
We observe that this is a slightly better estimate than Theorem 6 (an $L^1(\mathbb{T}^d)$ norm instead of the larger $L^2(\mathbb{T}^d)$ norm); this is maybe to be expected since one would assume that stronger estimates become available for the regular grid. We will also show that this is the best possible bound 
in terms of these $L^p$-spaces. It is an interesting question whether this bound ($L^1$ instead of $L^2$) is also true for the Kronecker sequence (Theorem 6). More generally, one could ask whether there
is a sequence $(x_n)_{n=1}^{\infty}$ that uniformly attains the same error estimate as Theorem 7.

\subsection{Other manifolds.} Nothing about our approach is particularly tied to the torus $\mathbb{T}^d$. 
Indeed, the main inequality
$$ W_2\left( \frac{1}{N} \sum_{k=1}^{N} \delta_{x_k} , dx \right) \lesssim \inf_{t > 0}  ~~ \sqrt{t} + \left( \sum_{k \in \mathbb{Z}^d \atop k \neq 0} \frac{e^{-\|k\|^2 t}}{\|k\|^2} \left| \frac{1}{N} \sum_{n=1}^{N} e^{2\pi i \left\langle k, x_n \right\rangle} \right|^2\right)^{1/2}  $$
can be generalized just as easily to other manifolds. Let us fix a manifold $(M,g)$ and use $\phi_k$ denote the sequence of Laplacian eigenfunctions
$$ -\Delta \phi_k = \lambda_k \phi_k.$$
We assume that $\phi_0 = 1$ is the trivial (constant) eigenfunction and that they are normalized to $\|\phi_k\|_{L^2} = 1$. Then the inequality (see  \cite{stein0}) assumes the form
$$ W_2\left( \frac{1}{N} \sum_{k=1}^{N} \delta_{x_k} , dx \right) \lesssim_M \inf_{t > 0} ~~ \sqrt{t} + \left( \sum_{k =1}^{\infty} \frac{e^{- 2 \lambda_k t}}{\lambda_k} \left| \frac{1}{N} \sum_{n=1}^{N} \phi_k(x_n) \right|^2\right)^{1/2} .$$
For most manifolds, we do not have an explicit expression for the eigenfunctions $\phi_k$ and the inequality is thus of limited use. There exists a substitute inequality in cases where the Green's function $G(x, y)$ or good estimates for it are known \cite{stein3}. However, the Laplacian eigenfunctions are completely explicit on the sphere and are simply the classical spherical harmonics that have already been frequently used to define notions of discrepancy on the sphere (see e.g. 
\cite{freeden, grabner00, grabner0, grabner, nar}). 
We believe that our notion can be a useful addition. As an example of its usefulness, we give the general version of the result above.

\begin{theorem} Let $(M,g)$ be a compact manifold without boundary, normalized to have volume 1, and let $f:\mathbb{T} \rightarrow \mathbb{R}$ be differentiable. Then, for some constant $c_M > 0$ depending only on the manifold, we have
$$ \left|  \int_{\mathbb{M}} f(x) dx - \frac{1}{N} \sum_{k=1}^{N} f(x_k) \right| \leq c_{M} \inf_{t > 0}\left[ \sqrt{t} \| \nabla f \|_{L^{\infty}} + \left\| e^{t \Delta} \sum_{k=1}^{N}{ \delta_{x_k}} \right\|_{\dot H^{-1}} \| \nabla f\|_{L^2} \right]$$
\end{theorem}
Alternatively, rewriting the Sobolev norm in terms of the spectral expansion, we could also write the upper bound on the integration error as
$$\inf_{t > 0}\left[ \sqrt{t} \| \nabla f \|_{L^{\infty}} + \left( \sum_{k=1}^{\infty} \frac{e^{- 2\lambda_k t}}{\lambda_k} \left| \frac{1}{N} \sum_{n=1}^{N} \phi_k(x_n) \right|^2  \right)^{1/2}\| \nabla f\|_{L^2} \right].$$
One possible application is to estimate the error of points chosen randomly with respect to the volume measure $dx$. We observe, from $L^2-$orthogonality of the Laplacian eigenfunctions, that if $(x_n)_{n=1}^{N}$ are chosen independently at random, then
\begin{align*}
\mathbb{E} \sum_{k =1}^{\infty} \frac{e^{- 2\lambda_k t}}{\lambda_k} \left| \frac{1}{N} \sum_{n=1}^{N} \phi_k(x_n) \right|^2 &= \sum_{k =1}^{\infty} \frac{e^{- 2\lambda_k t}}{\lambda_k}  \frac{1}{N^2} \sum_{n, \ell=1}^{N} \mathbb{E} \phi_k(x_n) \phi_k(x_{\ell}) \\
&= \frac{1}{N} \sum_{k =1}^{\infty} \frac{e^{- 2\lambda_k t}}{\lambda_k}
\end{align*}
Weyl's Theorem implies that, on a compact $d-$dimensional manifold, $\lambda_k \sim k^{2/d}$. For example, on $d-$dimensional manifolds with $d \geq 3$, we have (using Lemma 1 from below), for $0 < t < 1/2$,
$$  \frac{1}{N} \sum_{k =1}^{\infty} \frac{e^{- 2\lambda_k t}}{\lambda_k} \lesssim_M \frac{1}{N} \sum_{ k =1 }^{\infty} \frac{e^{-  k^{2/d} t}}{k^{2/d}} \lesssim_d  \frac{1}{N}t^{-\frac{d-2}{2}}.$$
Minimizing in $t$ suggests the value
$$ t^{1/2} = \frac{1}{N^{1/d}} \left(  \frac{\|\nabla f\|_{L^2}}{ \|\nabla f\|_{L^{\infty}}} \right)^{2/d}$$
resulting in the `typical bound' for random points
$$ \left|  \int_{M} f(x) dx - \frac{1}{N} \sum_{k=1}^{N} f(x_k) \right| \lesssim \| \nabla f\|_{L^{\infty}}^{\frac{d-2}{d}}  \| \nabla f\|_{L^{2}}^{\frac{2}{d}} N^{-1/d}.$$
However, this is inferior to classical Monte-Carlo and thus perhaps not useful.

\section{Proofs}

\subsection{A recurring computation.} We collect a simple Lemma that will reappear in several different arguments.

\begin{lemma} We have, for $m + d \geq 1$, the estimate
$$ \sum_{k \in \mathbb{Z}^d \atop k \neq 0} e^{-\|k\|^2 t} \| k \|^m \lesssim_{m,d} t^{-\frac{m+d}{2}}.$$
If $m+d=0$, then we have, for $0 < t < 1/2$,
$$ \sum_{k \in \mathbb{Z}^d \atop k \neq 0} e^{-\|k\|^2 t} \| k \|^m \lesssim_{m,d}  \log{\left(\frac{1}{t}\right)}.$$
\end{lemma}

\begin{proof} By moving to polar coordinates noting that, for all $\ell \geq 1$,
$$ \# \left\{ k \in \mathbb{Z}^d \setminus \left\{0\right\}: \ell \leq \|k\| < \ell +1 \right\} \leq c_d \ell^{d-1},$$
we can reduce the sum to a one-dimensional quantity
$$ \sum_{k \in \mathbb{Z}^d \atop k \neq 0} e^{-\|k\|^2 t} \| k \|^m \lesssim_d \sum_{k \in \mathbb{Z} \atop k \neq 0} e^{-|k|^2 t} | k |^{m+d-1}.$$
If $m+d=0$, then we can easily bound the sum via
$$\sum_{k \in \mathbb{Z} \atop k \neq 0} e^{-|k|^2 t} | k |^{-1} \lesssim \int_{1}^{\infty}{ \frac{e^{-x^2 t}}{x} dx}$$
This integral is the complete gamma function and can be rewritten in terms of the exponential integral via
$$ \int_{1}^{\infty}{ \frac{e^{-x^2 t}}{x} dx} = - \frac{1}{2}\mbox{Ei}(-t) = \frac{1}{2}\int_{t}^{\infty}{\frac{e^{-x}}{x} dx}.$$
It is easy to see that
$$ \int_{t}^{\infty}{\frac{e^{-x}}{x} dx} \lesssim \int_{t}^{1}{\frac{1}{x} dx} + \int_{1}^{\infty}{\frac{e^{-x}}{x} dx} \lesssim \log{\left(\frac{1}{t}\right)}.$$
It remains to deal with the case $m+d \geq 1$ where we estimate the sum via a different integral. Note that
$$  \sum_{k \in \mathbb{Z} \atop k \neq 0} e^{-|k|^2 t} | k |^{m+d-1} \lesssim \int_{0}^{\infty}{ e^{-x^2 t} x^{m+d-1} dx} = c_{m+d} t^{-\frac{m+d}{2}}.$$
\end{proof}

\subsection{Random Walks: Proof of Theorem 4.}
\begin{proof} 
We have that the measure $\mu_k$ describing the distribution of the random walk after $k$ steps is given by
$$ \mu_k = \mu_{k-1} * \mu \qquad \mbox{where}~*~\mbox{denotes convolution}$$
and 
$$ \mu = \frac{1}{2} \delta_{\alpha} + \frac{1}{2} \delta_{-\alpha}.$$
Therefore
$$ |\widehat{\mu_k}(\ell)| = |\widehat{\mu}(\ell)|^{k} = |\cos{(2\pi \ell \alpha)}|^k.$$
Using Peyre's estimate, we reduce the problem to estimating the sum
$$ W_2(\mu_k, dx) \leq \left( \sum_{\ell \in \mathbb{Z} \atop \ell \neq 0}{ \frac{|\cos{(2\pi \ell \alpha)}|^{2k}}{\ell^2} } \right)^{1/2}.$$
We use, as we often do, that $\ell \alpha$ cannot be close to an integer for many values of $\ell$. More precisely, we define the $k$ sets $$ I_j = \left\{ \ell \in  \mathbb{Z} \setminus \left\{ 0 \right\}:  \frac{j}{k} \leq \left\{ \ell \alpha\right\} \leq  \frac{j+1}{k}  \right\} \quad \mbox{for} \quad 0 \leq j \leq k-1.$$
Since $\alpha$ is badly approximable, we have that two distinct elements $\ell_1, \ell_2 \in I_j$ satisfy $| \ell_1 - \ell_2| \gtrsim_{\alpha} k.$
We can now write
$$ \sum_{\ell \in \mathbb{Z} \atop \ell \neq 0}{ \frac{|\cos{(2\pi \ell \alpha)}|^{2k}}{\ell^2}} = \sum_{j=0}^{k-1} \sum_{\ell \in I_j}{ \frac{|\cos{(2\pi \ell \alpha)}|^{2k}}{\ell^2}}.$$
We have
\begin{align*}
  \sum_{\ell \in I_j}{ \frac{|\cos{(2\pi \ell \alpha)}|^{2k}}{\ell^2}}  &\lesssim    \max_{x \in I_j} |\cos{(2\pi x)}|^{2k} \sum_{\ell \in I_j}{ \frac{1}{\ell^2} }.
\end{align*}
However, the smallest element in $I_j$ is $\gtrsim_{\alpha} k/(j+1)$ and any two consecutive elements are $\gtrsim_{\alpha} k$ separated implying that
$$ \sum_{\ell \in I_j}{ \frac{1}{\ell^2} } \lesssim_{\alpha} \sum_{h=0}^{\infty}{\frac{1}{(k/(j+1) + h k)^2}} \lesssim_{\alpha} \frac{(j+1)^2}{k^2}.$$

However, we also have
 $$\max_{x \in I_j} |\cos{(2\pi x)}|^{2k} \leq \left(1 - \left( \frac{\min\left\{j, k-j\right\}}{k}\right)^2 \right)^{2k}.$$
 By symmetry, it suffices to sum $j$ up to $k/2$. We then obtain
\begin{align*}
 \sum_{0 \leq j \leq k/2} \max_{x \in I_j} |\cos{(2\pi x)}|^{2k} \sum_{\ell \in I_j}{ \frac{1}{\ell^2} }&\lesssim \sum_{0 \leq j \leq k/2} \left(1 - \frac{j^2}{k^2}\right)^{2k} \frac{(j+1)^2}{k^2}\\
&\lesssim k \int_{0}^{1}{ \left(1 - x^2\right)^{2k} x^2 dx} \lesssim \frac{1}{\sqrt{k}}.
 \end{align*}
\end{proof}

\subsection{Kronecker sequences: Proof of Theorem 5}
\begin{proof}
Let us consider the Kronecker sequence 
$$ x_n = (n \alpha_1, n \alpha_2, n \alpha_3, \dots, n \alpha_d) ~\mbox{mod}~1.$$
We assume that $\alpha$ is badly approximable, which means that, for some universal constant $c_{\alpha} > 0$ and all integers $q \neq 0$, we have
$$ \max_{1 \leq j \leq d}{ \| \alpha_j q\|} \geq \frac{c_{\alpha}}{q^{1/d}},$$
where $\left\| \cdot \right\|$ is the distance to the nearest integer.
Khintchine's transference principle (see, for example, the textbook of Schmidt \cite{schm3}) states that $\alpha$
is badly approximable if and only if the linear form induced by $\alpha$ is badly approximable, i.e. if for all $0 \neq k \in \mathbb{Z}^d$
$$ \left\| \left\langle k, \alpha \right\rangle\right\| \geq \frac{c_{\alpha}}{\|k\|^{d}},$$
where $\left\| \cdot \right\|$ is the distance to the nearest integer and $c_{\alpha}$ is a universal constant. This is the property we are going to use.
Observe that, abbreviating 
$$ \mu = \frac{1}{N} \sum_{k=1}^{N}{ \delta_{x_k}},$$
then, arguing via the geometric series,
\begin{align*}
|\widehat{\mu}(k)| = \frac{1}{N} \left|\sum_{\ell=1}^{N}{ e^{2\pi i \left\langle k, x_{\ell} \right\rangle}}\right| = \frac{1}{N} \left|\sum_{\ell=1}^{N}{ e^{2\pi i \ell \left\langle k, \alpha \right\rangle}}\right| \leq  \frac{2}{N}\frac{1}{  \left\| \left\langle k, \alpha \right\rangle \right\|},
\end{align*}
where $\left\| \left\langle k, \alpha \right\rangle \right\|$ is the distance to the nearest integer.
We are left with estimating
$$ W_2(\mu, dx) \leq \inf_{t > 0} \left[\sqrt{t} + \frac{2}{N}\left( \sum_{k \neq 0} \frac{e^{-\|k\|^2 t}}{\|k\|^2}\frac{1}{ \left\|\left\langle k, \alpha \right\rangle\right\|^2}  \right)^{1/2}\right].$$
We split frequencies into dyadic scales and first estimate
$$ \sum_{2^{\ell} \leq \|k\| \leq 2^{\ell + 1}}  \frac{1}{ \left\|\left\langle k, \alpha \right\rangle\right\|^2}.$$
Clearly, for any $k_1 \neq k_2$ in this dyadic scale, we have 
$$\left| \left\langle k_1 - k_2, \alpha \right\rangle \right| \gtrsim_{\alpha} \|k_1-k_2\|^{-d} \geq 2^{-\ell d}.$$
This means that these $\sim 2^{\ell d}$ terms are roughly evenly spread and we have
$$ \sum_{2^{\ell} \leq \|k\| \leq 2^{\ell + 1}}\frac{1}{ \left\|\left\langle k, \alpha \right\rangle\right\|^2} \lesssim_{\alpha} \sum_{h=1}^{2^{\ell \cdot d}} \frac{1}{(h 2^{-\ell \cdot d})^2 } \lesssim 2^{2 \ell \cdot d}.$$
This shows that the typical size of such a term (of which there are $2^{\ell \cdot d}$) is $2^{\ell \cdot d}$ and thus we can estimate a dyadic block by increasing the multiplier as in
\begin{align*}
 \sum_{2^{\ell} \leq \|k\| \leq 2^{\ell + 1}} \frac{e^{-\|k\|^2 t}}{\|k\|^2}\frac{1}{ \left\|\left\langle k, \alpha \right\rangle\right\|^2}  &\leq 
 \left( \max_{2^{\ell} \leq \|k\| \leq 2^{\ell + 1}} \frac{e^{-\|k\|^2 t}}{\|k\|^2} \right) \sum_{2^{\ell} \leq \|k\| \leq 2^{\ell + 1}} \frac{1}{ \left\|\left\langle k, \alpha \right\rangle\right\|^2} \\
  &\lesssim_{\alpha}
 \left( \max_{2^{\ell} \leq \|k\| \leq 2^{\ell + 1}} \frac{e^{-\|k\|^2 t}}{\|k\|^2} \right) 2^{2 \ell \cdot d}\\
 &\lesssim_d  \left( \max_{2^{\ell} \leq \|k\| \leq 2^{\ell + 1}} \frac{e^{-\|k\|^2 t}}{\|k\|^2} \right)  \sum_{2^{\ell} \leq \|k\| \leq 2^{\ell + 1}} 2^{\ell \cdot d} \\
&\leq
\sum_{2^{\ell} \leq \|k\| \leq 2^{\ell + 1}} \frac{e^{-\|k\|^2 (t/2)}}{(\|k\|/2)^2} 2^{\ell \cdot d}.
\end{align*}
This, in turn, can be rewritten as the kind of sum already studied above since
$$ \sum_{2^{\ell} \leq \|k\| \leq 2^{\ell + 1}} \frac{e^{-\|k\|^2 (t/2)}}{(\|k\|/2)^2} 2^{\ell \cdot d} \lesssim_d \sum_{2^{\ell} \leq \|k\| \leq 2^{\ell + 1}} \frac{e^{-\|k\|^2 (t/2)}}{\|k\|^2} \|k\|^{d}$$
Altogether, using the Lemma above as well as $d \geq 2$,
$$\sum_{k \in \mathbb{Z}^d \atop k \neq 0} \frac{e^{-\|k\|^2 t}}{\|k\|^2}\frac{1}{ \left\|\left\langle k, \alpha \right\rangle\right\|^2} \lesssim \sum_{k \in \mathbb{Z}^d \atop k \neq 0} e^{-\|k\|^2 t} \|k\|^{d-2}  \lesssim \frac{1}{t^{d-1}}.$$
Therefore
$$ W_2(\mu, dx) \lesssim_{\alpha} \sqrt{t} + \frac{2}{N} \frac{1}{t^{\frac{d-1}{2}}}$$
which implies, for the choice $t=N^{-2/d}$ that
$$ W_2(\mu, dx) \lesssim_{\alpha} \frac{1}{N^{1/d}}.$$
\end{proof}

\subsection{Numerical Integration: Proof of Theorem 7}
The proof is based on a simple Poincar\'e-type inequality for Lipschitz functions vanishing at a fixed point. 
\begin{lemma}
Let $f:[0,1]^d \rightarrow \mathbb{R}$ be differentiable and assume that $$f(1/2, 1/2, \dots, 1/2) = 0.$$ Then we have the estimate
$$ \left| \int_{[0,1]^d}{f(x) dx} \right| \leq c_d   \| \nabla f\|_{L^{\infty}}^{\frac{d-1}{d}}  \| \nabla f\|_{L^{1}}^{\frac{1}{d}}.$$
\end{lemma}
The inequality is not new and follows from the combination of two known results. The inequality 
$$ \left| \int_{[0,1]^d}{f(x) dx} \right| \lesssim \int_{[0,1]^d} \frac{|\nabla f|}{|x|^{d-1}} dx $$
is used as a first step in the proof of Morrey's inequality (see Evans \cite[\S 5.6.2]{evans}). This is now combined with an interpolation estimate: it is easy to see that the function $g(x) = |x|^{1-d}$ is contained in the Lorentz space $L^{\frac{d}{d-1}, \infty}$. Thus, by the Holder inequality in Lorentz spaces due to O'Neil \cite{oneil}, we have
$$  \int_{[0,1]^d} \frac{|\nabla f|}{|x|^{d-1}} dx \lesssim \|f\|_{L^{d,1}}.$$
We recall the definition of the $L^{d,1}$ norm and use the H\"older inequality to obtain
\begin{align*}
\|f\|_{L^{d,1}} &= d\cdot  \| \lambda \cdot |\left\{ \left|f\right| > \lambda \right\}|^{1/d} \|_{L^1(\frac{d \lambda}{\lambda})} \\
&= d\cdot \int_{0}^{\infty} |\left\{ \left|f\right| > \lambda \right\}|^{1/d} d \lambda \\
&=  d\cdot \int_{0}^{\|f\|_{L^{\infty}}} |\left\{ \left|f\right| > \lambda \right\}|^{1/d} d \lambda \\
&\leq d\cdot \|f\|_{L^{\infty}}^{\frac{d-1}{d}} \cdot \left( \int_{0}^{\infty} \left|\left\{ |f| > \lambda \right\}\right| d\lambda \right)^{\frac{1}{d}} = d\cdot\|f\|_{L^{\infty}}^{\frac{d-1}{d}} \| f\|_{L^1}^{\frac{1}{d}}.
\end{align*}
Using this simple statement, we can now prove Theorem 7.

\begin{proof}[Proof of Theorem 7.] The proof of Theorem 7 follows easily from the Lemma which we apply, in isolation, to each fundamental cell of size $N^{-1/d}$. Rescaling the inequality in Lemma 2 then shows that for any such box $B=[0, N^{-1/d}]^d$, we have
$$ \left| \int_{B}{f(x) dx-\frac{1}{N} f(x_k)} \right| \leq  \frac{c_d}{N}    \| \nabla f\|_{L^{\infty}(B)}^{\frac{d-1}{d}}  \| \nabla f\|_{L^{1}(B)}^{\frac{1}{d}}.$$
Summing over all boxes leads to
\begin{align*}
 \left|  \int_{[0,1]^d} f(x) dx - \frac{1}{N} \sum_{k=1}^{N} f(x_k) \right| &\lesssim   \frac{\| \nabla f\|^{\frac{d-1}{d}}_{L^{\infty}}}{N}  \sum_{B} \| \nabla f\|^{\frac{1}{d}}_{L^{1}}  \\
   &\lesssim    \frac{\| \nabla f\|^{\frac{d-1}{d}}_{L^{\infty}}}{N}  \left(\sum_{B} \| \nabla f\|_{L^{1}}\right)^{1/d} \left( \sum_{B}1\right)^{\frac{d-1}{d}}\\
   &=  \frac{\| \nabla f\|^{\frac{d-1}{d}}_{L^{\infty}}}{N} \| \nabla f\|^{1/d}_{L^1} N^{\frac{d-1}{d}}\\
   &\leq \| \nabla f\|^{\frac{d-1}{d}}_{L^{\infty}}\| \nabla f\|^{1/d}_{L^1} N^{-1/d}.
\end{align*}
\end{proof}
We emphasize that the argument by itself actually yields a slightly stronger result in terms of local $L^1-$norms over $N^{-1/d}-$boxes
 $$\left|  \int_{[0,1]^d} f(x) dx - \frac{1}{N} \sum_{k=1}^{N} f(x_k) \right| \lesssim   \frac{\| \nabla f\|^{\frac{d-1}{d}}_{L^{\infty}}}{N}  \sum_{B} \| \nabla f\|^{\frac{1}{d}}_{L^{1}}.$$
\textbf{Optimality.} We quickly construct an example showing that our result is optimal. Let us consider the Lipschitz function 
$$ f(x) = \min\left\{\varepsilon, \min_{1 \leq i \leq N}{ \|x - x_i\|} \right\}.$$
We only work in the regime where $\varepsilon \ll N^{-1/d}$ in which case
we see that 
$$  \left|  \int_{[0,1]^d} f(x) dx - \frac{1}{N} \sum_{k=1}^{N} f(x_k) \right| =    \int_{[0,1]^d} f(x) dx \geq \varepsilon \left(1 - c_d N \varepsilon^d\right)$$
while also observing that 
$$ \|\nabla f\|_{L^{\infty}}=1 \qquad \mbox{and} \qquad \| \nabla f\|_{L^1} \sim N \varepsilon^d.$$
By letting $\varepsilon \rightarrow 0$, we see that our estimate has the optimal exponents.

\subsection{A General Manifold Result: Proof of Theorem 8}
\begin{proof} The proof combines two estimates. We first replace the discrete measure
$$ \mu = \frac{1}{N}\sum_{k=1}^{N}{\delta_{x_k}}$$
by the smoothed measure $e^{t\Delta} \mu$. The second step of the argument is merely a duality estimate (or, alternatively, an application of the Cauchy-Schwarz inequality). The first step is comprised of the estimate
$$ \left| \int_{M} f d\mu - \int_{M} f   e^{t\Delta} \mu d x \right| \lesssim \sqrt{t} \left\| \nabla f \right\|_{L^{\infty}},$$
which can be understood in at least two different ways. We describe both of them. The first case is physical: we interpret the heat equation as a process
 that transports a Dirac measure to a nearby neighborhood. The physical scaling is that within $t$ units of time, the mass is transported roughly distance $\sqrt{t}$.
However, the effect of transporting mass is naturally aligned to the setting of a Lipschitz function since
$$  \left| \int_{M} f d\mu - \int_{M} f d \nu \right| \leq \| \nabla f \|_{L^{\infty}} W_1(\mu, \nu).$$
(This inequality becomes an equation in one dimension and is known as Monge-Kantorovich duality in that setting, see e.g. \cite{villani}). However, it is known
that (see e.g. \cite{stein0})
$$ W_1(\mu, e^{t\Delta} \mu) \lesssim_{M} \sqrt{t} \| \mu\|_{L^1}$$
and since $\mu$ is normalized to be a probability measure, we obtain $\| \mu\|_{L^1} = 1$ and the desired estimate.
The second step is more explicit. We introduce the heat kernel $p_t(x,y)$ as the solution of the heat equation started with the measure $\delta_x$ and run up to time $t$
and then evaluted in $y$. Then it follows from conservation of mass that
$$ \int_{M}{ p_t(x,y) dy} = 1$$
and the mean-value theorem implies
\begin{align*}
 \left| \int_{M} f(x) d\mu - \int_{M} f(x) e^{t\Delta} \mu dx \right| &= \left| \int_{M} \frac{1}{N} \sum_{k=1}^{N}{  \left( p_t(x_k, y) f(y) - f(x_k) \right)} dy\right|\\
 &\leq  \frac{1}{N} \sum_{k=1}^{N}{ \left| \int_{M}  p_t(x_k, y) f(y) - f(x_k) dy\right|} \\
  &\leq  \frac{1}{N} \sum_{k=1}^{N}{ \left| \int_{M}  p_t(x_k, y) f(y) - p_t(x_k,y) f(x_k) dy\right|} \\
   &\leq   \frac{1}{N} \sum_{k=1}^{N}{  \int_{M} \| \nabla f\|_{L^{\infty}} p_t(x_k, y) |x_k -y| dy}\\
   &\leq\| \nabla f\|_{L^{\infty}}   \max_{x \in M}   \int_{M} p_t(x, y) |x -y| dy.
 \end{align*}
However, the last term can be controlled using Aronson's estimate
$$ 
p_{t}(x,y)\leq\frac{c_{1}}{t^{n/2}}\exp\left(-\frac{|x-y|^{2}}{c_{2}t}\right),\quad\forall t>0,x,y\in M,
$$
where the constant $c_1, c_2$ depend only on the manifold. A simple computation then shows (see e.g. \cite{stein0}) that
$$   \int_{M} p_t(x_k, y) |x_k -y| dy \lesssim_{M} \sqrt{t}.$$
We now come to the final part of the argument. It remains to estimate the error
 $$ \left| \int_{M} f(x) dx - \int_{M} f(x) e^{t\Delta} \mu~  dx\right| \qquad \mbox{from above}.$$
We interpret this as an inner product
 $$ \left| \int_{M} f(x) dx - \int_{M} f(x) e^{t\Delta} \mu~ dx \right| = \left| \left\langle f, e^{t \Delta}\mu - 1\right\rangle \right|.$$
 A duality argument now shows that
 $$  \left| \left\langle f, e^{t \Delta}\mu - 1\right\rangle \right| \leq  \| f\|_{\dot H^{1}} \| e^{t \Delta}\mu \|_{\dot H^{-1}}$$
 which is the desired result. One could also avoid the language of functional analysis and estimate, after noticing that $e^{t\Delta} \mu - 1$
 has mean value 0 and that $\phi_k$ has mean value 0 for $k \geq 1$,
 \begin{align*}
  \left| \left\langle f, e^{t \Delta}\mu - 1\right\rangle \right| &=\left| \sum_{k=0}^{\infty}{ \left\langle f, \phi_k\right\rangle  \left\langle e^{t\Delta}\mu -1, \phi_k\right\rangle} \right|\\
   &=\left| \sum_{k=1}^{\infty}{ \lambda_k^{1/2} \left\langle f, \phi_k\right\rangle \lambda_k^{-1/2} \left\langle e^{t\Delta}\mu -1, \phi_k\right\rangle} \right| \\
   &\leq    \left(  \sum_{k=1}^{\infty}{ \lambda_k^{} \left\langle f, \phi_k\right\rangle^2}   \right)^{1/2}     \left(  \sum_{k=1}^{\infty}{ \lambda_k^{-1} \left\langle e^{t\Delta}\mu -1, \phi_k\right\rangle^2}\right)^{1/2} \\
      &=    \left(  \sum_{k=1}^{\infty}{ \lambda_k^{} \left\langle f, \phi_k\right\rangle^2}   \right)^{1/2}     \left(  \sum_{k=1}^{\infty}{ \lambda_k^{-1} \left\langle e^{t\Delta}\mu, \phi_k\right\rangle^2}\right)^{1/2}.
  \end{align*}
  As for the first term, we observe that 
$$
  \sum_{k=1}^{\infty}{ \lambda_k \left\langle f, \phi_k\right\rangle^2}   = \int_{M} (-\Delta f)f dx  = \int_{M} |\nabla f|^2 dx = \|\nabla f\|_{L^2}^2.
$$
As for the second sum, we observe that, using the self-adjointness of the heat propagator and the fact that $\phi_k$ is an eigenfunction of the Laplacian
$$ \left\langle e^{t\Delta}\mu, \phi_k\right\rangle = \left\langle \mu, e^{t\Delta}\phi_k\right\rangle = e^{-\lambda_k t} \left\langle \mu, \phi_k\right\rangle.$$

This then results in
$$  \left(  \sum_{k=1}^{\infty}{ \lambda_k^{-1} \left\langle e^{t\Delta}\mu, \phi_k\right\rangle^2}\right)^{1/2} = 
 \left(  \sum_{k=1}^{\infty}{ \frac{e^{-2 \lambda_k t}}{ \lambda_k^{}} \left\langle \mu, \phi_k\right\rangle^2}\right)^{1/2}$$
and concludes the desired result.
\end{proof}

\subsection{Integration Error of Kronecker Sequences: Proof of Theorem 6}
\begin{proof} Having proven Theorem 8, we can now outline a proof of Theorem 6 which follows quite easily by combining several of our existing arguments. We will make use of the inequality
$$ \left|  \int_{\mathbb{M}} f(x) dx - \frac{1}{N} \sum_{k=1}^{N} f(x_k) \right| \leq c_{M} \inf_{t > 0}\left[ \sqrt{t} \| \nabla f \|_{L^{\infty}} + \left\| e^{t \Delta} \sum_{k=1}^{N}{ \delta_{x_k}} \right\|_{\dot H^{-1}} \| \nabla f\|_{L^2} \right]$$
in the special case where the manifold is given by $M = \mathbb{T}^d$ and the set of points is given by
$$ x_n = (n \alpha_1, \dots, n \alpha_d) ~\mbox{mod}~1$$
where $\alpha$ is badly approximable. The only quantity that requires computation is the $\dot H^{-1}$ norm. This, however, was already done in the proof of Theorem 5 where we used that
$$
W_2(\mu, dx) \lesssim_d \sqrt{t} + W_2(e^{t\Delta} \mu, dx) 
$$
and then estimated that
$$ W_2(e^{t\Delta} \mu, dx)  \lesssim_{\alpha, d} \frac{1}{N} \frac{1}{t^{\frac{d-1}{2}}}.$$
This results in
$$ \left|  \int_{\mathbb{M}} f(x) dx - \frac{1}{N} \sum_{k=1}^{N} f(x_k) \right| \lesssim_{d,\alpha} \inf_{t > 0}\left[ \sqrt{t} \| \nabla f \|_{L^{\infty}} +\frac{1}{N} \frac{1}{t^{\frac{d-1}{2}}} \|\nabla f\|_{L^2} \right].$$
We set
$$ t = \frac{1}{N^{2/d}} \frac{\| \nabla f\|_{L^2}^{2/d}}{\| \nabla f\|_{L^{\infty}}^{2/d}}.$$
This results in
 $$ \left|  \int_{\mathbb{M}} f(x) dx - \frac{1}{N} \sum_{k=1}^{N} f(x_k) \right| \lesssim_{d, \alpha} \frac{1}{N^{1/d}} \|\nabla f\|_{L^2}^{\frac{1}{d}} \| \nabla f\|_{L^{\infty}}^{\frac{d-1}{d}}.$$
\end{proof}

\textbf{Acknowledgments.} We are grateful to the anonymous referee for the careful reading and helpful suggestions.

\end{document}